\documentclass[11 pt]{amsart}

\usepackage[utf8]{inputenc}
\usepackage{tikz, tikz-3dplot}
\usepackage{amsmath, amsthm, amssymb,graphics,esint }
\usepackage{amsmath,amsthm,amssymb,amscd,color, xcolor,mathtools,url,tikz}
\usepackage{bbm}
\usepackage{pgfplots}

\newtheorem{thm}{Theorem}[section]

 \newtheorem{lem}{Lemma}[section]

 \theoremstyle{definition}
 
 \theoremstyle{remark}

 \numberwithin{equation}{section}

\newcommand{\dd}{\mathop{}\!\mathrm{d}}

\author{Charles Collot}
\address{CY Cergy Paris University. Laboratoire AGM, 2 avenue Adolphe Chauvin
95302 Cergy-Pontoise}
\email{charles.collot@cyu.fr}

\author{Pierre Germain} 
\address{Department of Mathematics, Huxley building, South Kensington campus, Imperial College London, London SW7 2AZ, United Kingdom}
\email{pgermain@ic.ac.uk}

\author{Eliot Pacherie}
\address{CNRS and CY Cergy Paris University. Laboratoire AGM, 2 avenue Adolphe Chauvin
95302 Cergy-Pontoise}
\email{eliot.pacherie@cyu.fr}

\begin{document}

\title[Absence of embedded spectrum for linearized nonlinear Schr\"odinger equations]{Absence of embedded spectrum for nonlinear Schr\"odinger equations linearized around one dimensional ground
states}

\maketitle

\begin{abstract}
    We consider the nonlinear Schrödinger equation in dimension one for a large class of nonlinearities. We show that ground states do not have embedded eigenvalues in the essential spectrum of their linearized operators.
\end{abstract}

\section{Introduction}

\subsection{The nonlinear Schr\"odinger equation linearized around a solitary wave}

We consider the nonlinear Schr{\"o}dinger equation
\begin{equation}
\label{NLS}
\tag{NLS}
i \partial_t u + \partial_x^2 u + F (| u |^2) u = 0 
\end{equation}
on $\mathbb{R} \times \mathbb{R}$ where the nonlinearity $F \in C^1
(\mathbb{R}^+, \mathbb{R})$ satisfies $F (0) = 0$. Standing waves are solutions of the form
\[ u (x, t) = e^{i \mu t} Q (x) \]
for some $\mu > 0$, which are indeed solutions if and only if
\begin{equation}\label{elliptic-standing-wave} \partial_x^2 Q - \mu Q + F (| Q |^2) Q = 0. 
\end{equation}
Then, for $p, \gamma, y \in \mathbb{R}$, the Galilean, phase and translation symmetries of \eqref{NLS} give the family of traveling waves
$$
e^{i( p x +( \mu-p^2) t + \gamma)} Q (x-2p t - y).
$$
Such solutions are important to understand the behaviour of solutions to \eqref{NLS}, in particular because of the soliton resolution conjecture which asserts that general solutions should decompose asymptotically in time into the sum of decoupled solitons. Of particular importance are ground state solutions to \eqref{elliptic-standing-wave}, which are those, without loss of generality by translation invariance, that are even, positive, exponentially decaying, and such that $Q'<0$ for $x>0$. They are expected to be more stable than other solutions to \eqref{elliptic-standing-wave} which are called excited states. Ground states exist under fairly general assumptions on $F$. For example, if it is assumed to be smooth, with a non-degenerate local minimum at zero, see \cite{BL}.

Perturbing the standing wave solution by $\varphi$ through the ansatz
\[ u (x, t) = e^{i \mu t} \left[ Q (x) + \varphi (x, t) \right], \]
the equation on $\varphi$ becomes
\[ i \partial_t \varphi + L_Q (\varphi) + O (\varphi^2) = 0 \]
where
\[ L_Q (\varphi) = \partial_x^2 \varphi - \mu \varphi + F (| Q |^2) \varphi +
   2\mathfrak{R}\mathfrak{e} (Q \bar{\varphi}) F' (| Q |^2) Q. \]
In this paper, we focus on the linearized equation
\[ i \partial_t \varphi + L_Q (\varphi) = 0 \]
in the case where $Q$ is an aforementioned ground state. The operator $L_Q$ is not complex linear but
simply real linear, and to turn this problem into a complex linear one, we
write this equation as
\[ i \partial_t \left(\begin{array}{c}
     \varphi\\
     \bar{\varphi}
   \end{array}\right) +\mathcal{H} \left(\begin{array}{c}
     \varphi\\
     \bar{\varphi}
   \end{array}\right) = 0 \]
where
\[ \mathcal{H} := \mathcal{H}_0 + V \]
with
\[ \mathcal{H}_0 := \left(\begin{array}{cc}
     \partial_x^2 - \mu & 0\\
     0 & - \partial_x^2 + \mu
   \end{array}\right) \]
and
\begin{equation}
  V := \left(\begin{array}{cc}
    F (Q^2) + F' (Q^2) Q^2 & F' (Q^2) Q^2\\
    - F' (Q^2) Q^2 & -F (Q^2) - F' (Q^2) Q^2
  \end{array}\right) . \label{potV}
\end{equation}

\subsection{The embedded spectrum of $\mathcal{H}$}
Since $Q$ is exponentially decaying, so is the potential
$V$, and by classical arguments, the essential spectrum of $\mathcal{H}$
is the same as the one of $\mathcal{H}_0$, namely
\[ \sigma_e (\mathcal{H}) = \sigma_e (\mathcal{H}_0) = ( - \infty, - \mu] \cup
   [\mu, + \infty ). \]

In order to prove linear as well as nonlinear asymptotic stability of the solitary wave $e^{i\mu t} Q(x)$, the first step is to exclude the existence of nonzero eigenvalues. Point spectrum can be ruled out by a number of methods, but it has remained an open problem to prove that there is no embedded spectrum for the operator $\mathcal{H}$ defined above, except in very specific cases. This problem is stated in \cite{CGNT,ErdoganSchlag}, and in the latter reference it is also conjectured that no embedded spectrum appears if $Q$ is a ground state.

To illustrate the difficulty of excluding embedded spectrum, it is helpful to think of the one-dimensional case. If one considers a scalar Schr\"odinger operator $-\partial_x^2 + V$, where $V$ decays appropriately, it a basic exercise in ODE theory to rule out the existence of embedded spectrum. But if one turns to the matrix Schr\"odinger operator
$$
\mathcal{G} = \left(\begin{array}{cc}
     \partial_x^2 - \mu & 0\\
     0 & - \partial_x^2 + \mu
   \end{array}\right) + \left(\begin{array}{cc}
    V & 0\\
     0 & -V   \end{array}\right),
$$
it is such that $\mathcal{G} \begin{pmatrix} f \\ 0 \end{pmatrix} = \lambda \begin{pmatrix} f \\ 0 \end{pmatrix}$ as soon as $(\partial_x^2 - \mu + V ) f = \lambda f$. Since this equation can have nontrivial solutions for $\lambda < -\mu$ (if $V$ is sufficiently negative), this means that embedded spectrum can exist for general matrix Schr\"odinger operators (this example appears in \cite{KriegerSchlag}).

To the best of our knowledge, the only cases where embedded spectrum could be excluded for $\mathcal{H}$ correspond are the following. First, for pure power nonlinearity in dimension 1, where $Q$ is given by an explicit formula, tools from complex analysis become available \cite{KriegerSchlag, Perelman}. Second, when a quadratic form akin to a virial identity is positive on a suitable space, as is checked numerically for various nonlinearities and dimensions \cite{Marzuola_2011,Asad-Simpson}.

The argument and the conclusions of the present paper apply in far greater generality, although still in dimension $1$.

\subsection{Main result}
Our goal is to exclude the existence of eigenvalues of
$\mathcal{H}$ in the essential spectrum. Given $\lambda \in \mathbb{R}$, we define the eigenspace
\[ \mathcal{E}_{\lambda} = \left\{ \left(\begin{array}{c}
     f\\
     g
   \end{array}\right) \in (L^2 (\mathbb{R}, \mathbb{R}))^2, \; (\mathcal{H}-
   \lambda) \left(\begin{array}{c}
     f\\
     g
   \end{array}\right) = 0 \right\}. \] 
Our main result is as follows.

\begin{thm}
  \label{main}
In dimension $1$, if $Q$ is an even, positive solution of (\ref{elliptic-standing-wave}), and
  converges to $0$ exponentially fast with $Q' < 0$ on $(0,+\infty)$, then for any $\lambda \in  \sigma_e (\mathcal{H})$,
  \[ \mathcal{E}_{\lambda} = \{ 0 \} .
  \]
\end{thm}

\subsection{Implications for the asymptotic stability of solitary waves}

In \cite{BuslaevPerelman,BuslaevSulem,CollotGermain,Cuccagna}, asymptotic stability of ground states of \eqref{NLS} is proved under the assumption that embedded spectrum is absent. This assumption can be removed by Theorem \ref{main}.

\subsection*{Acknowledgements} PG was supported by the Simons Foundation through the Wave Turbulence Collaboration, and by a Wolfson fellowship of the Royal Society.

This result is part of the ERC starting grant project FloWAS that has received funding from the European Research Council (ERC) under the Horizon Europe research and innovation program (Grant agreement No. 101117820). C. Collot was supported by the CY Initiative of Excellence Grant ”Investissements d’Avenir” ANR-16-IDEX-0008 via Labex MME-DII, the grant ”Chaire Professeur Junior” ANR-22-CPJ2-0018-01 of the French National Research Agency and the grant BOURGEONS ANR-23-CE40-0014-01 of the French National Research Agency.

\section{Proof of Theorem \ref{main}}

We consider now $Q$ satisfying the hypotheses of Theorem \ref{main}. Since $Q>0$, decays exponantially fast and solves $Q'' - \mu Q + F (Q^2) Q = 0$ with $\mu > 0$ and $F
(0) = 0$, by classical arguments, there exists $c_0 > 0$ such
that
\begin{equation}
  Q (x) = c_0 e^{-\sqrt{\mu} x} (1 + o_{x \rightarrow
  + \infty} (1)),Q'(x)=-\sqrt{\mu}c_0 e^{-\sqrt{\mu} x} (1+o_{x \rightarrow
  + \infty} (1)). \label{kanine}
\end{equation}
Indeed, $Q$ is solution to the linear ODE $f''-\mu f+F(Q^2)f = 0$ which has two solutions, one growing and the other decaying exponentially fast at rate $\sqrt{\mu}$. $Q$ then must be a linear combination of them, and since it decays, it is colinear to the exponentially decaying solution.
This solution can be constructed via the integral
formulation
\begin{equation}
  f (x) = e^{- \sqrt{\mu} x} - e^{\sqrt{\mu} x} \int_x^{+ \infty} e^{- 2
  \sqrt{\mu} t} \int_t^{+ \infty} F (Q^2) (s) e^{\sqrt{\mu} s} f (s) d s d t
  \label{avb}
\end{equation}
and a fixed point argument for $x \geqslant R > 0$ large enough to show that
the second part of the right hand side is a contraction for the norm $\left\|
f e^{\sqrt{\mu} .} \right\|_{C^1 ([R, + \infty [)}$. Differentiating it, we
deduce the equivalent when $x \rightarrow + \infty$ of $Q' (x) = c_0 f'(x)$.

\subsection{Behavior near $x = + \infty$ of solutions}

\begin{lem}
  \label{fiberoptic}Given $\lambda \in  \sigma_e (\mathcal{H})$, any nonzero solution
  of
  \[ (\mathcal{H}- \lambda) \left(\begin{array}{c}
       f\\
       g
     \end{array}\right) = 0 \]
  that is in $L^2 (\mathbb{R}, \mathbb{R})$ must satisfy, for some $c_* \neq 0$,

\begin{align*}
& \left(\begin{array}{c}
       f\\
       g
     \end{array}\right) = c_{\ast} e^{-\sqrt{\mu + \lambda} x} \left( \left(\begin{array}{c}
       1\\
       0
     \end{array}\right) + o_{x \rightarrow + \infty} (1) \right)
      \qquad \mbox{if $\lambda \geq \mu$} \\
& \left(\begin{array}{c}
       f\\
       g
     \end{array}\right) = c_{\ast} e^{-\sqrt{\mu - \lambda} x}
     \left( \left(\begin{array}{c}
       0\\
       1
     \end{array}\right) + o_{x \rightarrow + \infty} (1) \right) \qquad \mbox{if $\lambda \leq - \mu$}.
\end{align*}

\end{lem}

\begin{proof}
  We recall that the equation $\mathcal{H}- \lambda = 0$ is
  \[ \left(\begin{array}{cc}
       \partial_x^2 - \mu - \lambda & 0\\
       0 & - (\partial_x^2 - \mu + \lambda)
     \end{array}\right) \left(\begin{array}{c}
       f\\
       g
     \end{array}\right) = - V \left(\begin{array}{c}
       f\\
       g
     \end{array}\right) . \]
  
  \
  
\noindent  1) If $\lambda > \mu > 0$, if we consider the equation with $V = 0$, the
  radial solutions of $\partial_x^2 - \mu - \lambda = 0$ are $e^{-\sqrt{\mu + \lambda} x} $ and $e^{\sqrt{\mu + \lambda} x} $, while the solutions of $\partial_x^2 - \mu +
  \lambda$ are $\cos \left( \sqrt{\lambda - \mu } x \right)$ and
  $\sin \left( \sqrt{\lambda - \mu } x \right)$.
  
  Since $V$ decays exponentially fast by (\ref{kanine}), we construct by a
  fixed point argument four solutions of $\mathcal{H}- \lambda = 0$ near $x =
  + \infty$ which satisfy respectively

  \begin{eqnarray*}
    \left(\begin{array}{c}
      f_1\\
      g_1
    \end{array}\right) & = & e^{-\sqrt{\mu + \lambda } x}\left( \left(\begin{array}{c}
       1\\
       0
     \end{array}\right) + o_{x \rightarrow + \infty} (1) \right)\\
    \left(\begin{array}{c}
      f_2\\
      g_2
    \end{array}\right) & = & e^{\sqrt{\mu +  \lambda} x} \left( \left(\begin{array}{c}
       1\\
       0
     \end{array}\right) + o_{x \rightarrow + \infty} (1) \right)\\
    \left(\begin{array}{c}
      f_3\\
      g_3
    \end{array}\right) & = & \cos \left( \sqrt{ \lambda - \mu }\, x
    \right)\left( \left(\begin{array}{c}
       0\\
       1
     \end{array}\right) + o_{x \rightarrow + \infty} (1) \right)\\
    \left(\begin{array}{c}
      f_4\\
      g_4
    \end{array}\right) & = & \sin \left( \sqrt{ \lambda - \mu }\, x
    \right)\left( \left(\begin{array}{c}
       0\\
       1
     \end{array}\right) + o_{x \rightarrow + \infty} (1) \right) .
  \end{eqnarray*}
  They can indeed be constructed by a fixed point similar to (\ref{avb}), replacing $F(Q^2)$ by $V$, $\mu$ by $\mu+\lambda$ or $\mu-\lambda$ (with the possibility for their square root to be imaginary) and for vector valued functions, using the decay of $V$ to close the fixed point argument for $r$ large enough.

\
  
  Since $\mathcal{H}- \lambda = 0$ is a fourth order ODE,
  any solution must be a linear combination of these functions. Looking at
  their behavior near $x = + \infty$, the only possibility for such a
  combination to be in $L^2 (\mathbb{R}, \mathbb{R})$ is for
  the solution to be colinear to $\left(\begin{array}{c}
    f_1\\
    g_1
  \end{array}\right)$.
  
  \
  
\noindent  2)  If $\lambda < - \mu$, the solutions of $\partial_x^2 - \mu - \lambda$ are then
  $\cos \left( \sqrt{| \lambda + \mu |} x \right)$ and $\sin
  \left( \sqrt{| \mu + \lambda |} x \right)$, while the solutions of $\partial_x^2 -
  \mu + \lambda$ are $e^{-\sqrt{\mu -  \lambda } x}$
  and $e^{\sqrt{\mu -  \lambda } x}$. We conclude
  as in the first case, but now the exponentially decaying function is at the limit on the
  vector $\left(\begin{array}{c}
    0\\
    1
  \end{array}\right)$.

\noindent 3) If $\lambda = \mu$ or $\lambda = - \mu$, instead of the two oscillating functions, it is $1$ and $x$ that are solutions for $V=0$. We can still construct two solutions of $\mathcal{H}-\lambda$ with these behaviors near $x = +\infty$ by the same fixed point argument, and any linear combination of them do not belong in $L^2$, leading to the same conclusion.
  
\end{proof}

\subsection{Inversion results involving $L_+$ and $L_-$}

We define
\begin{align*}
& L_- = - \partial_x^2 + \mu - F (Q^2) \\
& L_+ = - \partial_x^2 + \mu - F (Q^2) - 2 F' (Q^2) Q.
\end{align*}
\begin{lem}(Inversion of $L_-$)
  \label{bonesintosand}
 Consider the equation $L_- (u) = v$ on functions assuming that there exists $\varepsilon > 0$ such that
  \[ \left\| (| u | + | v |) (x) e^{\left( \sqrt{\mu} + \varepsilon \right) x}
     \right\|_{L^{\infty} ([1, + \infty [, \mathbb{R})} < + \infty . \]
  Then
  \[ u (x) = - Q (x) \int_x^{+ \infty} \frac{1}{Q^2 (t) } \int_t^{+
     \infty} Q (s)  v (s) \dd s. \]
\end{lem}

\begin{proof}
  We have $L_- (Q) = 0$. We look for a solution of the form $u = Q h$ leading
  to
  \[ - \frac{1}{Q} (Q^2 h')' = - Q \left( h'' +  2
     \frac{Q'}{Q}  h' \right) = v \]
  hence
  \[ (Q^2 h')' = - Q  v. \]
  We deduce that
  \[ Q^2(x) h' (x) = c_2 + \int_x^{+ \infty} Q (t) v (t) \dd t
  \]
  for some $c_2 \in \mathbb{R}$, leading to
  \[ h (x) = c_1 + c_2 \int_1^x \frac{\dd t}{Q^2 (t) } - \int_x^{+
     \infty} \frac{1}{Q^2 (t) } \int_t^{+ \infty} Q (s) v
     (s) \dd s \]
  for some $c_1 \in \mathbb{R}$. Indeed, with the decay near $x = + \infty$ of
  $Q$ and $v$, these integral are all well defined. In other words,
  \begin{eqnarray*}
u (x) = c_1 Q (x) + c_2 Q (x) \int_1^x \frac{\dd t}{Q^2 (t) } - Q
    (x) \int_x^{+ \infty} \frac{1}{Q^2 (t)} \int_t^{+ \infty} Q (s)
    v (s) \dd s \dd t.
  \end{eqnarray*}
  Now, with the decay at $+ \infty$ of both $Q$ and $v$, $Q (x) \int_1^x
  \frac{d t}{Q^2 (t) }$ grows exponentially fast, and all the other
  quantities in this equality converges to $0$, hence $c_2 = 0$. But then $Q(x) = c_0
  e^{-\sqrt{\mu} x} (1 + o_{x \rightarrow + \infty}
  (1))$ is the only remaining term that does not decay at least like $e^{-
  \left( \sqrt{\mu} + \varepsilon \right) x}$ for some $\varepsilon > 0$, hence
  $c_1 = 0$.
\end{proof}

\begin{lem}(Inversion of $L_+$) \label{InvL+}
Assuming that $Q'<0$ on $\mathbb{R}$, two functions satisfying $L_+ (u) = v$ under the constraint that there exists $\varepsilon > 0$ such that
  \[ \left\| (| u | + | v |) (x) e^{\left( \sqrt{\mu} + \varepsilon \right) x}
     \right\|_{L^{\infty} ([1, + \infty [, \mathbb{R})} < + \infty. \]
are such that  \[ u(x) = - Q' (x) \int_x^{+ \infty} \frac{1}{Q'^2 (t) }
     \int_t^{+ \infty} Q' (s)  v (s) \dd s. \]
\end{lem}

Since $L_+(Q')=0$, the proof is identical to that of Lemma \ref{bonesintosand}, replacing $L_-$ by $L_+$ and $Q$ by $Q'$.

\subsection{End of the proof of Theorem \ref{main}}

The equation
\[ (\mathcal{H}- \lambda) \left(\begin{array}{c}
     f\\
     g
   \end{array}\right) = 0 \]
can be written as the system
\[ \left\{\begin{array}{l}
     L_+ (f + g) = - \lambda (f - g)\\
     L_- (f - g) = - \lambda (f + g) .
   \end{array}\right. \]
We define $u = f + g, v = f - g$. Suppose that $f, g \in L^2 (\mathbb{R},
\mathbb{R}) \setminus \{ (0,0) \}$. Then, by Lemma \ref{fiberoptic}, there exists $\varepsilon > 0$
such that $\left\| (| u | + | v |) (x) e^{\left( \sqrt{\mu} + \varepsilon
\right) x} \right\|_{L^{\infty} ([1, + \infty [, \mathbb{R})} < + \infty$, and
then by Lemma \ref{bonesintosand}, we have
\begin{align}
 & u (x) = \lambda Q' (x) \int_x^{+ \infty} \frac{1}{Q'^2 (t)
  } \int_t^{+ \infty} Q' (s)  v (s) \dd s \label{camo}\\
 & v (x) = \lambda Q (x) \int_x^{+ \infty} \frac{1}{Q^2 (t) }
  \int_t^{+ \infty} Q (s) u (s) \dd s \label{crooked}
\end{align}
where $Q>0$ and $Q'<0$ do not vanish.

\

To start with, we consider the case $\lambda \geqslant \mu > 0$. Then, near $x = +
\infty$, by Lemma \ref{fiberoptic} we have
\[ \left(\begin{array}{c}
     u\\
     v
   \end{array}\right) = c_{\ast}  e^{-\sqrt{ \mu  + 
   \lambda } x} \left( \left(\begin{array}{c}
       1\\
       1
     \end{array}\right) + o_{x \rightarrow + \infty} (1) \right) \]
for some $c_{\ast} \neq 0$; we can even assume that $c_{\ast} = 1$.
This implies that near $x = + \infty$, we have $u > 0, v > 0$. Let us show that this implies that $u, v > 0$ on
$\mathbb{R}^{+ \ast}$. Indeed, denoting by $x_{\ast} > 0$ the largest value
such that $u (x_{\ast}) v (x_{\ast}) = 0$, then since $u, v > 0$ on
$]x_{\ast}, + \infty [$ and $Q,-Q'>0,  $ we get a contradiction by the
formulas (\ref{camo}) and (\ref{crooked}).

We have therefore shown that $u, v > 0$ on $\mathbb{R}^{+ \ast}$.

Using $Q'(0)=0$ since $Q$ is even, (\ref{crooked}) and the continuity of $v$ at $x=0$, we have in particular that $v
(0) > 0, v' (0) < 0$.
By Lemma \ref{fiberoptic}, we have that
\[ \left(\begin{array}{c}
     u\\
     v
   \end{array}\right) = c_{\ast} e^{-\sqrt{\mu + | \lambda |} (- x)  }\left( \left(\begin{array}{c}
       1\\
       0
     \end{array}\right) + o_{x \rightarrow + \infty} (1) \right) \]
for some $c_{\ast} \in \mathbb{R}^{\ast}$ that we can no longer normalize to $1$. However, by similar arguments using Lemmas \ref{bonesintosand} and \ref{InvL+}, we have (now
coming from $x = - \infty$) that
\[ c_{\ast} v (0) > 0, c_{\ast} v' (0) > 0, \]
leading to $v (0) v' (0) > 0$, which is in contradiction with $v (0) > 0, v'
(0) < 0$.

\

For the case $\lambda \leqslant - \mu < 0$, by Lemma \ref{fiberoptic} we have now
\[ \left(\begin{array}{c}
     u\\
     v
   \end{array}\right) = c_{\ast} e^{-\sqrt{ \mu  - 
   \lambda } x } \left( \left(\begin{array}{c}
       1\\
       -1
     \end{array}\right) + o_{x \rightarrow + \infty} (1) \right). \]
Normalizing with $c_\ast = 1$
and using the fact that $\lambda < 0$, we claim that we can show here that $u > 0$ and now $v < 0$ on $\mathbb{R}^{+ \ast}$, and we reach a contradiction similarly.

\bibliographystyle{abbrv}

\bibliography{references}

\end{document}